\documentclass[11pt]{amsart}

\usepackage{amsmath, amssymb}
\usepackage{amscd}
\usepackage{verbatim}
\newtheorem{definition}{Definition}[section]
\newtheorem{theorem}[definition]{Theorem}
\newtheorem{lemma}[definition]{Lemma}

\newtheorem{proposition}[definition]{Proposition}
\theoremstyle{definition}

\newcommand{\la}{\left\langle}
\newcommand{\ra}{\right\rangle}

\newcommand{\bbC}{\mathbb{C}}

\newcommand{\adcom}{{\pi(A)}^{\prime \prime}}

\begin{document}

%%%%%%%%%%%%%%%%%%%%%%%%%%%%%%%%%%%%%%%%%%%%%

\title[Barycentric Decompositions in the Space of WE]{Barycentric Decompositions in the Space of Weak Expectations}

\author[A.~Bhattacharya]{Angshuman Bhattacharya}
\address{Department of Mathematics, IISER Bhopal, MP 462066, India}
\email{angshu@iiserb.ac.in}

\author[C. J.~Kulkarni]{Chaitanya J. Kulkarni}
\address{Department of Mathematics, IISER Bhopal, MP 462066, India}
\email{kulkarni18@iiserb.ac.in}


%%%%%%%%%%%%%%%%%%%%%%%%%%%%%%%%%%%%%%%%%%%%%%%

\begin{abstract}
The space of weak expectations for a given representation of a (unital) C*-algebra is a compact convex subset of the unit ball of completely bounded maps in the BW topology, when it is non-empty. An application of the classical Choquet theory gives a barycentric decomposition of a weak expectation in that set. However, to complete the barycentric picture, one needs to know the extreme points of the compact convex set in question. In this article, we explicitly identify the set of extreme points of the space of weak expectations for a given representation, using operator theoretic techniques. 
\end{abstract}

\maketitle

%%%%%%%%%%%%%%%%%%%%%%%%%%%%%%%%%%%%%%%%%%%%%%%	

\section{Introduction} \label{sec;I}

Let $A$ be a unital C*-algebra and $\pi: A\rightarrow B(H)$ be an arbitrary but fixed non-degenerate representation of $A$ on a Hilbert space $H$. A unital completely positive map $\theta : B(H) \rightarrow \adcom$ is called a weak expectation of $\pi$ if $$\theta (\pi (a))=\pi(a)$$ for all $a\in A$. We fix the following notation for this article:
\begin{eqnarray*}
\rm{WE}(\pi) &:=& \{\theta :B(H) \rightarrow \adcom ~|~ \theta~ \mbox{is a weak expectation of}~ \pi\} \\
\rm{CB}(B(H), \adcom) &:=& \{\theta :B(H) \rightarrow \adcom ~|~ \theta~ \mbox{is completely bounded} \} 
\end{eqnarray*}  
The set $\rm{CB}(B(H), \adcom)$ is a locally convex topological vector space with respect to the BW-topology. Let $\rm{CP}_1(B(H), \adcom)$ be the set of all completely positive maps from $B(H)$ to $\adcom$ with norm less than or equal to 1. This set is a compact, convex subset of the topological vector space $\rm{CB}(B(H), \adcom)$. The set $\rm{WE}(\pi)$, when \textit{non-empty}, is a closed convex subset of $\rm{CP}_1(B(H), \adcom)$ in the BW-topology and hence compact. 

While it is not necessary in general, that $\rm{WE}(\pi)\neq \emptyset$, there are several well known cases in which it is indeed so. Some immediate examples would be: $\rm{WE}(\pi)\neq \emptyset$ for any irreducible representation $\pi$ and the GNS representation $\pi=\pi_\tau$ of an \textit{amenable trace} $\tau$ on a C*-algebra, see \cite[Chapter 6]{BO}. Another example would be the inclusion representation of any separable C*-algebra $A\hookrightarrow B(H)$, which is \textit{relatively weakly injective} in $B(H)$. See \cite[Lemma 3.4]{Kir93} for the abundant existence of such C*-algebras. Henceforth, in this article we will always assume $\rm{WE}(\pi)\neq \emptyset$ for the fixed representation $\pi$.

Recall that, the classical Choquet-Bishop-de Leeuw theorem of integral representation of a point in a compact convex affine space by a measure on the space states that \cite{Phlp, OB1}:

\begin{theorem} [Choquet-Bishop-de Leeuw] \label{cbl}
Let $X$ be a compact convex subset of a locally convex topological vector space $E$. Let $x_0 \in X$. Then there exist a Baire measure $\mu_{x_0}$ pseudo supported on the extreme points $\rm{ext}(X)$ of $X$ such that for any continuous linear functional $f: E \rightarrow \bbC$ one has: $$f(x_0)=\int f ~\rm{d}\mu_{x_0}.$$ 
\end{theorem}

In the above theorem, we need $E$ to be a locally convex topological vector space to just ensure the existence of sufficiently many linear functionals in $E^*$ that separate points in $X$. Now, given the representation $\pi:A\rightarrow B(H)$ of a unital C*-algebra $A$ and the set $\rm{WE}(\pi)$, we see that the classical Choquet-Bishop-de Leeuve theorem is applicable for the locally convex topological vector space $E = \rm{CB}(B(H), \adcom)$ (equipped with the BW-topology) and $X = \rm{WE}(\pi)$ as the compact convex set in question. We consider the following separating family of continuous linear functionals on $E = \rm{CB}(B(H), \adcom)$, which separates the points in $X = \rm{WE}(\pi)$. The family is given by:
\begin{equation*}
\{ f^S_{(h, k)}: E  \rightarrow \bbC ~ | ~ f^S_{(h, k)}(\phi) := \la \phi(S)h, k \ra; S \in B(H), h,k \in H, \phi \in E \}.
\end{equation*}

This gives a \textit{barycentric decomposition} of any weak expectation in $\rm{WE}(\pi)$. However, to complete this decomposition picture, one needs to identify the extreme points of the set $\rm{WE}(\pi)$ explicitly. This is precisely what this article accomplishes: \textit{using operator theoretic techniques, we identify which weak expectations (of $\pi$) constitute the set of extreme points in} $\rm{WE}(\pi)$. 

Using some technical results from Section \ref{sec;EP}, we explicitly identify the extreme points of $\rm{WE}(\pi)$ in Theorem \ref{mt2}. While the basic idea of the proof is inspired by Arveson's \cite{A69} operator theoretic proof of describing the extreme points of completely positive maps into $B(H)$, there are several subtle differences in our constructions, owing to the constraints levied by the properties of a weak expectation. Further, we obtain an equivalent formulation of extreme points in terms of operators on certain Hilbert modules (Theorem \ref{mt2}(4)), using Paschke's Hilbert module dilation of completely positive maps, \cite{WP73}.

\section{Preliminaries} \label{secP}

In this section we briefly recall the results useful for us in this article. First and foremost, we recall the Stinespring dilation theorem. 

\begin{theorem}\cite[Theorem 4.1]{VP}\label{thm;SD}
Let $A$ be a unital C*-algebra and $H$ be a Hilbert space. Let $\phi : A \rightarrow B(H)$ be a completely positive map. Then there exists a Hilbert space $K$, a bounded linear map $V : H \rightarrow K$ and a unital *-homomorphism $\rho : A \rightarrow B(K)$ such that
\begin{equation*}
\phi(a) = V^*\rho(a)V \hspace{10pt} \text{for all $a \in A$}. 
\end{equation*} 
Moreover, the set $\{ \rho(a)Vh : a \in A, h \in H \}$ spans a dense subspace of $K$.
\end{theorem}

We call the triple $(\rho, V, K)$ from Theorem \ref{thm;SD} as the minimal Stinespring representation for $\phi$. For a given completely positive map $\phi$, the minimal Stinespring representation is unique upto an unitary equivalence  (see \cite[Proposition 4.2]{VP}).

Let $A$ be a unital C*-algebra and $\rm{CP}(A, B(H))$ be the set of all completely positive maps from $A$ to $B(H)$ where $H$ is a Hilbert space. The partial order on $\rm{CP}(A, B(H))$, defined by $\phi_1 \leq \phi_2$, if $\phi_2 - \phi_1 \in \rm{CP}(A, B(H))$. For $\phi \in \rm{CP}(A, B(H))$, then $$[0, \phi] = \{ \theta \in CP(A, B(H)) : \theta \leq \phi \}.$$ Let $\phi (a) = V^* \rho (a) V$ be the \textit{minimal} Stinespring dilation of $\phi$, where $V$ is a bounded map from $H$ to $K$ and $\rho$ is representation of $A$ on $K$. For each operator $T \in \rho(A)^\prime$, define $\phi_T := V^* \rho T V$ to be a map from $A$ to $B(H)$. Arveson proved the following Radon--Nikodym type theorem for completely positive maps in \cite{A69}:

\begin{theorem} \cite[Theorem 1.4.2]{A69} \label{rnd}
The map $T \mapsto \phi_T = V^* \rho T V$ is an affine order isomorphism of the partially ordered convex set of operators $\{T \in \rho(A)^\prime : 0 \leq T \leq  1_K \}$ onto $[0, \phi]$.
\end{theorem}

We shall also use Paschke's dilation of completely positive maps on Hilbert modules. We quickly recall a few results about the Paschke's dilation of completely positive maps that are required in this article.

Let $A$ and $B$ be C*-algebras and $X$ be a Hilbert $B$-module. Suppose $\sigma: A \rightarrow \mathcal{P}(X)$ be *-representation of $A$ on $X$ where $\mathcal{P}(X)$ is the set of all \textit{adjointable operators} on $X$ (see \cite{WP73}, for the definition of an adjointable operator on $X$). The adjointable operators $\mathcal{P}(X)$ form a C*-algebra \cite{WP73}. For $e \in X$, the map $\phi : A \rightarrow B$ defined by
\begin{eqnarray*}
\phi(a) := \langle \sigma(a)e, e \rangle
\end{eqnarray*}
is a completely positive map. Paschke proved in \cite{WP73} that every completely positive map from $A$ into $B$ arises this way. The following theorem is referred to the Paschke's dilation

\begin{theorem} \cite[Theorem 5.2]{WP73} \label{pd}
Let $A$ and $B$ be unital C*-algebras and $\phi \colon A \rightarrow B$ be a completely positive map. Then there exists a Hilbert $B$-module $X$, a *-representation $\sigma \colon A \rightarrow \mathcal{P}(X)$ and $e \in X$ such that $\phi(a) = \langle \sigma(a)e, e \rangle$ for all $a \in A$. The set $\{ \sigma(a)(e.b) : a \in A, b \in B \}$ spans a dense subspace of $X$.
\end{theorem}

For a C*-algebra $B$ and a pre-Hilbert $B$-module $X$ define
\begin{eqnarray*}
X^\prime &:=& \{ \tau: X \rightarrow B ~|~ \tau ~ \text{is a bounded $B$-module map} \}.
\end{eqnarray*} 
For each $x \in X$ we have a map $\hat{x} \colon X \rightarrow B$ as $\hat{x}(y) := \langle y, x \rangle$. If $X^\prime = \hat{X} = \{\hat{x} ~|~ x\in X \}$, then $X$ is said to be self-dual. If $B$ is a von-Neumann algebra, then $X^\prime$ is also a Hilbert $B$-module and the $B$-valued inner product on $X$ extends to a $B$-valued inner product on $X^\prime$.

\begin{theorem} \cite[Theorem 3.2]{WP73} \label{xprime}
Let $B$ be a von-Neumann algebra and $X$ be a pre-Hilbert $B$-module. The $B$-valued inner product $\langle \cdot, \cdot \rangle$ on $X \times X$ extends to $X^\prime \times X^\prime$ in such a way as to make $X^\prime$ into a self-dual Hilbert $B$-module. Also, it satisfies $\langle \hat{x}, \tau \rangle = \tau(x)$ for $x \in X$, $\tau \in X^\prime$.
\end{theorem}

Since $X$ is a submodule of $X^\prime$, the next proposition extends an adjointable operator on $X$ to an adjointable operator on $X^\prime$:

\begin{proposition} \cite[Corollary 3.7]{WP73} \label{pext}
Let $B$ be a von-Neumann algebra and $X$ be a pre-Hilbert $B$-module, then each $T \in \mathcal{P}(X)$ extends to a unique $\tilde{T} \in \mathcal{P}(X^\prime)$. The map $T \mapsto \tilde{T}$ is a *-isomorphism of $\mathcal{P}(X)$ into $\mathcal{P}(X^\prime)$.
\end{proposition}

Let $A$ be a unital C*-algebra, $B$ be a von Neumann algebra and $\phi:A \rightarrow B$ is a completely positive map. By Theorem \ref{pd} and Proposition \ref{pext} one defines a *-representation $\tilde{\sigma}: A \rightarrow \mathcal{P}(X^\prime)$ by $\tilde{\sigma}(a) := \widetilde{\sigma(a)}$. For $T \in \mathcal{P}(X^\prime)$, define a linear map $\phi_T : A \rightarrow B$ by $\phi_T(a) := \langle T \tilde{\sigma}(a) \hat{e}, \hat{e} \rangle$ for $a \in A$. Then the following Radon--Nikodym type result is available in this Hilbert module setting:

\begin{theorem} \cite[Proposition 5.4]{WP73} \label{prnd}
The map $T \mapsto \phi_T$ is an affine order isomorphism of $\{ T \in  \tilde{\sigma}(A)^\prime \hspace{5pt}| \hspace{5pt} 0\leq T \leq 1_{X^\prime} \}$ onto $$[0, \phi] = \{ \theta \in CP(A, B) \hspace{5pt} | \hspace{5pt} \theta \leq \phi \}.$$
\end{theorem}

\section{Extreme points and barycentric decompositions} \label{sec;EP}
The main goal of this section is to characterize extreme points of the compact convex set $\rm{WE}(\pi)$ for an arbitrary but fixed representation $\pi$ of a unital C*-algebra $A$. Before embarking on the main result of this section, we prove two technical lemmas which are key to the arguments presented in the main result (Theorem \ref{mt2}). 

Let $\pi: A \rightarrow B(H)$ be a non-degenerate representation of a unital C*-algebra $A$ and $\phi: B(H) \rightarrow \pi(A)^{\prime\prime} \subseteq B(H)$ be a weak expectation of $\pi$. Let $\phi = V^* \rho V$ be the minimal Stinespring dilation of $\phi$, where $V$ is a bounded map from $H$ to $K$ and $\rho$ is representation of $B(H)$ on $K$. Here we briefly recall a few details and notations about the minimal Stinespring dilation which we will be using throughout the article. The Hilbert space $K$ is constructed by defining a bilinear function $\la \cdot ~, \cdot \ra$ on the algebraic tensor product $B(H) \otimes H$ as
\begin{equation*}
\la S_1 \otimes h_1, S_2 \otimes h_2 \ra := \la \phi(S_2^*S_1)h_1, h_2 \ra.
\end{equation*}
and extending linearly. Since $\phi$ is completely positive, $\la \cdot ~, \cdot \ra$ is positive semidefinite. Consider the null space
\begin{equation*}
N := \{u \in B(H) \otimes H :  \la u, u \ra = 0 \}.
\end{equation*} 
Let $K_0 = (B(H) \otimes H) / N$. Define a bilinear function on the quotient space $K_0$ by $\la \cdot ~, \cdot \ra$ (again by using the same notation) as
\begin{equation*}
\la u_1 + N, u_2 + N \ra := \la u_1, u_2 \ra.
\end{equation*}
Then $K_0$ is an inner product space with respect to $\la \cdot ~, \cdot \ra$ and the Hilbert space $K$ is obtained by completing $K_0$. The operator $V : H \rightarrow K$ is defined by 
\begin{equation*}
V(h) := 1_H \otimes h + N.
\end{equation*}
Finally the *-homomorphism $\rho : B(H) \rightarrow B(K)$ is defined on the dense subspace $K_0$ by
\begin{equation*}
\rho(S) \left (\sum_{i=1}^{k} S_i \otimes h_i + N \right ) := \sum_{i=1}^{k} SS_i \otimes h_i + N
\end{equation*}
and then extending the operator $\rho(S)$ to $K$.

Now we state the first technical lemma. Before this we clarify some notations. For any given $h \in H$, we see that $1_H \otimes h + N \in K$. Now, let $T\in B(K)$. We write the vector $T(1_H \otimes h + N) \in K$ in the form:
\begin{equation*}
T(1_H \otimes h + N) = \lim_{\alpha} \left( \sum_{i=1}^{n_\alpha} S^h_{\alpha_i} \otimes h_{\alpha_i} + N \right)
\end{equation*}
for some \textit{representative net} $\left \{ \sum\limits_{i=1}^{n_\alpha} S^h_{\alpha_i} \otimes h_{\alpha_i} + N \right \}_\alpha \subseteq K_0$.

\begin{lemma} \label{ml1}
Let $A$ be a unital C*-algebra and $\pi : A \rightarrow B(H)$ be a non-degenerate representation of $A$. Let $\phi \in \rm{WE}(\pi)$ and $V^* \rho V$ be the minimal Stinespring dilation of $\phi$ on the Hilbert space $K$. Suppose $\phi_T = V^* T \rho V \leq \phi$ where $T \in \rho(B(H))^\prime$ and $0 \leq T \leq  1_K$. Then there exists $0 < t <  1$ and $\theta \in \rm{WE}(\pi)$ such that $\phi_T = t \theta$ if and only if the following property is satisfied: for an arbitrary but fixed $h \in H$ and for any representation of the vector $T(1_H \otimes h + N)$ given by
\begin{equation*}
T(1_H \otimes h + N) = \lim_{\alpha} \left( \sum_{i=1}^{n_\alpha} S^h_{\alpha_i} \otimes h_{\alpha_i} + N \right)
\end{equation*}  
one has 
\begin{enumerate}
\item $\lim\limits_\alpha \sum\limits_{i=1}^{n_\alpha} \phi (S^h_{\alpha_i}) h_{\alpha_i} = th$  \hspace{3pt} \text{for some} \hspace{3pt} $0 < t < 1$;
\item for any $R \in \pi(A)^\prime$, the vector of the form $T(1_H \otimes R h + N)$ is represented by $T(1_H \otimes R h + N) = \underset{\alpha}{\lim} \sum\limits_{i=1}^{n_\alpha} S^h_{\alpha_i} \otimes R h_{\alpha_i} + N.$
\end{enumerate} 
\end{lemma}
\begin{proof}
Let $\phi \in \rm{WE}(\pi)$ and $V^* \rho V$ be the minimal Stinespring dilation of $\phi$. Suppose $\phi_T = V^* T \rho V \leq \phi$ where $T \in \rho(B(H))^\prime$ such that $0 \leq T \leq  1_K$. If there exists $0 < t < 1$ and $\theta \in \rm{WE}(\pi)$ such that $\phi_T = t \theta$, then $\phi_T(\pi(a)) = t \pi(a)$, for all $a \in A$. Therefore for an arbitrary but fixed $h \in H$ we have:
\begin{eqnarray*} 
\phi_T(\pi(a))h = V^*  \rho(\pi(a)) T V h &=& V^* \lim_{\alpha} \sum_{i=1}^{n_\alpha} \pi(a) S^h_{\alpha_i} \otimes h_{\alpha_i} + N  \\
&=&  \lim_{\alpha} \sum_{i=1}^{n_\alpha} \phi (\pi(a) S^h_{\alpha_i}) h_{\alpha_i} \\
&=& \pi(a) \lim_{\alpha} \sum_{i=1}^{n_\alpha} \phi (S^h_{\alpha_i}) h_{\alpha_i} \\
&=& \pi(a) t h.		
\end{eqnarray*} 
This is true for all $a \in A$ and as $\pi$ is a non-degenerate representation we get,
\begin{equation*}
\lim\limits_\alpha \sum\limits_{i=1}^{n_\alpha} \phi (S^h_{\alpha_i}) h_{\alpha_i} = t h.
\end{equation*}  
This proves the first equation.
	
Now we prove the second equation. Consider an operator $R$ in $\pi(A)^\prime$ and any $S$ in $B(H)$, then we have
\begin{eqnarray*}
V^*  \rho(S) T V R h &=& R V^*  \rho(S) T V h \\
&=&  R V^*  \rho(S) T (1_H \otimes h + N) \\
&=&  R V^* \lim_{\alpha} \sum_{i=1}^{n_\alpha} S S^h_{\alpha_i} \otimes h_{\alpha_i} + N \\
&=&  R \lim_{\alpha} V^* \sum_{i=1}^{n_\alpha} S S^h_{\alpha_i} \otimes h_{\alpha_i} + N \\
&=&  R \lim_{\alpha} \sum_{i=1}^{n_\alpha} \phi (S S^h_{\alpha_i}) h_{\alpha_i} \\
&=&  \lim_{\alpha} \sum_{i=1}^{n_\alpha} \phi (S S^h_{\alpha_i})  Rh_{\alpha_i} \\
&=&  \lim_{\alpha} V^* \sum_{i=1}^{n_\alpha} S S^h_{\alpha_i} \otimes Rh_{\alpha_i} + N .
\end{eqnarray*} 
Therefore, for $h^\prime \in H$, we have:
\begin{eqnarray*}
\left \langle V^* \rho(S) T V R h,~ h^\prime \right \rangle  &=& \left \langle V^*  \rho(S) T (1_H \otimes Rh + N), ~h^\prime \right \rangle \\
&=& \left \langle T (1_H \otimes Rh + N), ~S^* \otimes h^\prime + N \right \rangle 	
\end{eqnarray*}		
and
\begin{eqnarray*}
\left \langle R V^*  \rho(S) T V h, ~ h^\prime \right \rangle &=& \left \langle \lim_{\alpha} V^* \sum_{i=1}^{n_\alpha} S S^h_{\alpha_i} \otimes Rh_{\alpha_i} + N, ~ h^\prime \right \rangle \\
&=& \lim_{\alpha} \left \langle   \sum_{i=1}^{n_\alpha} S S^h_{\alpha_i} \otimes Rh_{\alpha_i} + N, ~ 1_H \otimes h^\prime + N \right \rangle \\
&=& \lim_{\alpha} \left \langle   \rho (S) \sum_{i=1}^{n_\alpha}  S^h_{\alpha_i} \otimes Rh_{\alpha_i} + N, ~ 1_H \otimes h^\prime + N \right \rangle \\
&=& \left \langle  \lim_{\alpha} \sum_{i=1}^{n_\alpha}  S^h_{\alpha_i} \otimes Rh_{\alpha_i} + N, ~ S^* \otimes h^\prime + N \right \rangle .		
\end{eqnarray*}
This is true for all $S \in B(H)$ and $h^\prime \in H$ hence, 
$$T(1_H \otimes R h + N) = \lim_{\alpha} \sum_{i=1}^{n_\alpha} S^h_{\alpha_i} \otimes R h_{\alpha_i} + N.$$

Now to prove the converse, let $T \in \rho(B(H))^\prime$ with $0 \leq T \leq  1_K.$ Suppose $h \in H$ be an arbitrary but fixed vector and $\left \{ \sum\limits_{i=1}^{n_\alpha} S^h_{\alpha_i} \otimes h_{\alpha_i} + N \right \}_\alpha$ be any representing net corresponding to $T(1_H \otimes h + N)$. If one has 		
\begin{enumerate}
\item $\underset{\alpha}{\lim} \sum\limits_{i=1}^{n_\alpha} \phi (S^h_{\alpha_i}) h_{\alpha_i} = th$ \hspace{3pt} \text{for some} \hspace{3pt} $0 < t < 1$;
\item for any $R \in \pi(A)^\prime$, the vector of the form $T(1_H \otimes R h + N)$ is represented by $T(1_H \otimes R h + N) = \underset{\alpha}{\lim} \sum\limits_{i=1}^{n_\alpha} S^h_{\alpha_i} \otimes R h_{\alpha_i} + N.$	
\end{enumerate}
Then we show that there exists $\theta \in \text{WE}(\pi)$ such that $\phi_T = t \theta$.

To this end, using equation (1) above we get, for $a \in A$ and $h \in H$
\begin{eqnarray*}
\phi_T (\pi(a))h = V^* \rho (\pi(a)) T Vh &=& V^* \lim_{\alpha} \sum_{i=1}^{n_\alpha} \pi(a) S^h_{\alpha_i} \otimes h_{\alpha_i} + N  \\
&=& \lim_{\alpha} \sum_{i=1}^{n_\alpha} \phi (\pi(a) S^h_{\alpha_i}) h_{\alpha_i} \\
&=& \pi(a) \lim_{\alpha} \sum_{i=1}^{n_\alpha} \phi (S^h_{\alpha_i}) h_{\alpha_i} \\
&=& \pi(a) t h. 
\end{eqnarray*}
Since, this is true for all $h \in H$, $\phi_T (\pi(a)) = t \pi(a)$. 		
	
Now we show that the range of $\phi_T$ is contained in $\pi(A)^{\prime\prime}$. Let $R$ and $S$ be any operators in $\pi(A)^\prime$ and $B(H)$ respectively. Let $h$ be an arbitrary but fixed vector in $H.$ Then using equation (2) above and the fact that the range of $\phi$ is contained in $\pi(A)^{\prime\prime}$ we get,
\begin{eqnarray*}
\phi_T (S) R h = V^* \rho (S) T V Rh &=&  V^* \rho (S) T (1_H \otimes Rh + N)\\
&=& V^* \rho (S) \lim_{\alpha} \sum_{i=1}^{n_\alpha} S^h_{\alpha_i} \otimes Rh_{\alpha_i} + N \\
&=& V^* \lim_{\alpha} \sum_{i=1}^{n_\alpha} S S^h_{\alpha_i} \otimes Rh_{\alpha_i} + N \\
&=& \lim_{\alpha} \sum_{i=1}^{n_\alpha} \phi (S S^h_{\alpha_i}) R h_{\alpha_i} \\
&=& R \lim_{\alpha} \sum_{i=1}^{n_\alpha} \phi (S S^h_{\alpha_i})  h_{\alpha_i}	
\end{eqnarray*}
and 
\begin{eqnarray*}
R \phi_T (S) h = R V^* \rho (S) T V h &=& R V^* \rho (S) T (1_H \otimes h + N)\\
&=& R V^* \rho (S) \lim_{\alpha} \sum_{i=1}^{n_\alpha} S^h_{\alpha_i} \otimes h_{\alpha_i} + N \\
&=& R V^* \lim_{\alpha} \sum_{i=1}^{n_\alpha} S S^h_{\alpha_i} \otimes
h_{\alpha_i} + N \\
&=& R \lim_{\alpha} \sum_{i=1}^{n_\alpha} \phi (S S^h_{\alpha_i})  h_{\alpha_i}.	
\end{eqnarray*}
Since $h \in H$ were arbitrarily chosen, we get
\begin{equation*}
\phi_T (S) R = R \phi_T (S).
\end{equation*}		
This is true for all $S \in B(H)$ and $R \in \pi(A)^\prime$. Therefore, the range of $\phi_T$ is contained in $\pi(A)^{\prime\prime}$.
	
Then defining $\theta := \frac{1}{t}\phi_T$ concludes the proof.
\end{proof}

Now we prove a similar result in the setting of Paschke's dilation of completely positive maps. Let $\pi : A \rightarrow B(H)$ be a non-degenerate representation of a unital C*-algebra $A$ and $\phi$ be a weak expectation of $\pi$. Taking $A = B(H)$ and $B = \adcom$ in Theorem \ref{pd} we get the triple $(X, \sigma, e)$ such that 
\begin{equation*}
\phi(S) = \langle \sigma(S)e, e \rangle
\end{equation*}
for all $S \in B(H)$. Here we provide some details and set some notations from the proof of Theorem~\ref{pd} which are useful in the remaining part of this article. We briefly review the construction of the Hilbert $\adcom$-module $X$, the *-representation $\sigma \colon B(H) \rightarrow \mathcal{P}(X)$ and the element $e \in X$. 

The Hilbert $\adcom$-module $X$ is constructed by using the algebraic tensor product $B(H) \otimes \adcom$, which is a right Hilbert $\adcom$-module with respect to the operation $(S \otimes T_1) T_2 = (S \otimes T_1 T_2)$, for all $S \in B(H)$ and $T_1, T_2  \in \adcom$. Define a conjugate bilinear function 
\begin{equation*}
[\cdot ~, \cdot] : \left ( B(H) \otimes \adcom \right ) \times \left ( B(H) \otimes \adcom \right ) \rightarrow \adcom 
\end{equation*} 
as
\begin{equation*}
\left [S_1 \otimes T_1, S_2 \otimes T_2 \right ] := T_2^* \phi(S_2^*S_1)T_1
\end{equation*} 
We get $[\cdot ~, \cdot]$ on $\left ( B(H) \otimes \adcom \right ) \times \left ( B(H) \otimes \adcom \right )$ by extending the above definition linearly. Since $\phi$ is completely positive, $[\cdot ~, \cdot]$ is positive semidefinite. Consider the null-submodule \begin{equation*}
M := \{x \in B(H) \otimes \adcom ~ | ~ [ x , x ] = 0 \}.
\end{equation*} 
Let $X_0 = \left (B(H) \otimes \adcom \right ) / M$. Define a $\adcom$-valued inner product on the quotient space $X_0$ by $[\cdot ~, \cdot]$ (again by using the same notation)
\begin{equation*}
[ x_1 + M, x_2 + M ] := [ x_1, x_2 ].
\end{equation*}
The quotient space $X_0$ is a pre-Hilbert $\adcom$-module with a $\adcom$-valued inner product $[\cdot ~, \cdot]$. Then the Hilbert $\adcom$-module $X$ is the completion of $X_0$ and the element $e \in X$ is given by $e = 1_H \otimes 1_H + M$. Finally the *-representation $\sigma \colon B(H) \rightarrow \mathcal{P}(X)$ is defined on the dense submodule $X_0$ by
\begin{equation*}
\sigma(S) \left (\sum_{i=1}^{k} S_i \otimes T_i + M \right ) := \sum_{i=1}^{k} SS_i \otimes T_i + M
\end{equation*}
and then extending $\sigma(S)$ uniquely to an operator in $\mathcal{P}(X)$.

Now we state the second technical lemma. This lemma is in the Hilbert module setting and is similar to Lemma \ref{ml1}. Before this we clarify some notations. For $\phi \in \rm{WE}(\pi)$, let $X, \sigma, e$ be as defined above. Correspondingly, we define $X^\prime, \tilde{\sigma}, \hat{e}$ as described in Section~\ref{secP}. For an arbitrary but fixed $T \in \mathcal{P}(X^\prime)$ we denote $T(\hat{e})$ as $\tau_T$. We use these notations to state the following lemma.

\begin{lemma} \label{ml2}
Let $T \in \tilde{\sigma}(B(H))^\prime$ with $0 \leq T \leq 1_{X^\prime}$. Suppose $\phi_T : B(H) \rightarrow \pi(A)^{\prime\prime}$ is given by $\phi_T (\cdot) := \langle T \tilde{\sigma}(\cdot) \hat{e}, \hspace{2pt} \hat{e} \rangle$. Then there exists $0 < t <  1$ and $\theta \in \text{WE}(\pi)$ such that $\phi_T = t \theta$ if and only if for all $a \in A$, $\tau_T(\pi(a) \otimes 1_H + M) = t \pi(a)$ for some $0 < t < 1$. 
\end{lemma}
\begin{proof}
We know from Theorem \ref{prnd} that $\phi_T$ is a completely positive map from $B(H)$  into $\pi(A)^{\prime\prime}$ with $\phi_T \leq \phi$. For any $a \in A$, we have
\begin{eqnarray*}
\phi_T(\pi(a)) &=& \langle T \tilde{\sigma}(\pi(a)) \hat{e},~\hat{e} \rangle \\
&=& \langle T \widetilde{\sigma(\pi(a))} \hat{e},~\hat{e} \rangle \\
&=& \langle T \widehat{\sigma(\pi(a))e},~\hat{e} \rangle \\
&=& \langle \widehat{\sigma(\pi(a))e},~\tau_T \rangle \\
&=& \tau_T(\pi(a) \otimes 1_H + M)	
\end{eqnarray*}
The last step follows from the Theorem \ref{xprime}.

Suppose there exists $0 < t <  1$ and $\theta \in \text{WE}(\pi)$ such that $\phi_T = t \theta$, then 
\begin{equation*}
\phi_T(\pi(a)) =  \tau_T(\pi(a) \otimes 1_H + M) = t \pi(a).
\end{equation*}

Conversely, if for all $a \in A$, $\tau_T(\pi(a) \otimes 1_H + M) = t \pi(a)$ for some $0 < t <  1$, then we get the result by defining $\theta := \frac{1}{t} \phi_T$.			
\end{proof}

Now we define a few sets which are useful in our analysis.

Let $\pi: A \rightarrow B(H)$ be a non-degenerate representation of a unital C*-algebra $A$. For $\phi \in \text{WE}(\pi)$, let $\phi=V^* \rho V$ be the minimal Stinespring dilation of $\phi$ on the Hilbert space $K$ and $(X, X^\prime, \sigma, \tilde{\sigma}, e, \hat{e})$ be the corresponding triple and extensions from the Paschke dilation of $\phi$ on the Hilbert module $X$. Then we define the following sets:\\

\begin{itemize}
\item $S_\phi := \{ T \in \rho (B(H))^\prime  ~ |~  0 \leq T \leq  1_K ;~ T ~ \rm{as~in~Lemma ~\ref{ml1}}\}$
\item $\mathcal{C}^S_\phi:= \{ \alpha_1 T_1 + \alpha_2 T_2 ~ | ~ \alpha_1, \alpha_2 \in [-1, 1]; ~ T_1, T_2 \in S_\phi \}$ 
\item $P_\phi := \{ T \in \tilde{\sigma}(B(H))^\prime ~ | ~ 0 \leq T \leq 1_{X^\prime}; ~ T ~ \rm{as~ in~ Lemma ~\ref{ml2}}\}$
\item $\mathcal{C}^P_\phi := \{ \alpha_1 T_1 + \alpha_2 T_2 ~ | ~ \alpha_1, \alpha_2 \in [-1, 1]; ~ T_1, T_2 \in P_\phi\}$\\	
\end{itemize}

Finally, we recall a definition due to Arveson from \cite{A69}:

\begin{definition}
Let $H$ be a Hilbert space and $H_0$ be a closed subspace of $H$. Let $P$ denote the projection of $H$ onto $H_0$, then, we say that $H_0$ is a faithful subspace of $H$ for a set of operators $A \subseteq B(H)$ containing $0$, if, for every $a \in A$, $PaP = 0$ implies $a =0$.  	
\end{definition}

Now we are ready to present the main result of this section.

\begin{theorem} \label{mt2}
Let $\pi: A \rightarrow B(H)$ be a non-degenerate representation of a unital C*-algebra $A$ and $\rm{WE}(\pi)\neq \emptyset$. Let $\phi=V^* \rho V \in \text{WE}(\pi)$ denote the minimal Stinespring dilation of $\phi$ on the Hilbert space $K$ and $\phi(\cdot) = \langle \sigma(\cdot)e, e \rangle$ denote the Paschke dilation of $\phi$. Then the following statements are equivalent:
\begin{enumerate}
\item The map $\phi$ is an extreme point of $\text{WE}(\pi)$.
\item If $\phi_0 \in [0, \phi] \cap t\text{WE}(\pi)$, then $\phi_0 = t\phi$ where, $0 < t < 1$ $( [0, \phi]$ is considered as a subset of $CP(B(H), \pi(A)^{\prime \prime}) )$.
\item The subspace $[\rho(\pi(A))VH]$ is a faithful subspace of $K$ for $\mathcal{C}^S_\phi$.  
\item If any $T \in \mathcal{C}^P_\phi$ satisfies $\langle T \tilde{\sigma}(\pi(a))\hat{e}, \hat{e} \rangle = 0 ~\forall~ a\in A$, then $T=0$. 
\end{enumerate}		
\end{theorem}

\begin{proof}
We proceed by showing that $(1)\Leftrightarrow (2), (3), (4)$.\\
	
$(1)\Leftrightarrow (2)$: Suppose $\phi$ is an extreme point of $\text{WE}(\pi)$ and $\phi_0 \in [0, \phi] \cap t\text{WE}(\pi)$. Then $\phi_0 = t \phi_1 \leq \phi$ for some $\phi_1 \in \text{WE}(\pi)$. Using Theorem \ref{rnd}, we get $\phi_0 = t\phi_1 = V^* T \rho V$ for $T \in \rho(B(H))^\prime$ with $0 \leq T \leq  1_K$. From Lemma \ref{ml1} we have for an arbitrary but fixed $h \in H$ and any representing net $\left \{ \sum\limits_{i=1}^{n_\alpha} S^h_{\alpha_i} \otimes h_{\alpha_i} + N \right \}_\alpha$ corresponding to $T(1_H \otimes h + N)$ one has 
\begin{enumerate}
\item $\lim\limits_\alpha \sum\limits_{i=1}^{n_\alpha} \phi (S^h_{\alpha_i}) h_{\alpha_i} = th$;
\item for any $R \in \pi(A)^\prime$, the vector of the form $T(1_H \otimes R h + N)$ is represented by $T(1_H \otimes R h + N) = \underset{\alpha}{\lim} \sum\limits_{i=1}^{n_\alpha} S^h_{\alpha_i} \otimes R h_{\alpha_i} + N$.		
\end{enumerate} 
	
Define $\phi_2 := V^* \frac{(1_K - T)}{1 - t} \rho V$, then
\begin{eqnarray*}
\phi_2(\pi(a))h &=& V^*  \rho(\pi(a)) \frac{(1_K - T)}{1 - t} V h \\
&=& \frac{V^*}{1 - t}  \rho(\pi(a)) \left (1_H \otimes h + N - \lim_{\alpha} \sum_{i=1}^{n_\alpha} S^h_{\alpha_i} \otimes h_{\alpha_i} + N \right )\\
&=&  \frac{V^*}{1 - t} \left (\pi(a) \otimes h + N - \lim_{\alpha} \sum_{i=1}^{n_\alpha} \pi(a) S^h_{\alpha_i} \otimes h_{\alpha_i} + N \right )\\
&=& \frac{1}{1 - t} \left (\phi(\pi(a))h - \lim_{\alpha} \sum_{i=1}^{n_\alpha} \phi \left (\pi(a) S^h_{\alpha_i} \right ) h_{\alpha_i} \right )\\
&=& \frac{1}{1 - t} \left ( \pi(a)h -t\pi(a)h \right ) = \pi(a)h. 
\end{eqnarray*} 
Since, this is true for any $h \in H$, we have $\phi_2 (\pi(a)) = \pi(a)$ for all $a \in A$.

Now we show that the range of $\phi_2$ is contained in $\pi(A)^{\prime\prime}$. If $R \in \pi(A)^\prime$, then for all $ S \in B(H)$
\begin{eqnarray*}
\phi_2 (S) R h &=& V^*  \rho(S) \frac{(1_K - T)}{1 - t} V R h \\
&=& \frac{V^*}{1 - t} \left (S \otimes Rh + N - \lim_{\alpha} \sum_{i=1}^{n_\alpha} SS^h_{\alpha_i} \otimes Rh_{\alpha_i} + N \right )\\
&=& \frac{1}{1 - t} \left (\phi(S)Rh - \lim_{\alpha} \sum_{i=1}^{n_\alpha} \phi (S S^h_{\alpha_i}) R h_{\alpha_i} \right ) \\
&=& \frac{R}{1 - t} \left (\phi(S)h - \lim_{\alpha} \sum_{i=1}^{n_\alpha} \phi (S S^h_{\alpha_i})  h_{\alpha_i} \right )
\end{eqnarray*}
and 		
\begin{eqnarray*}
R \phi_2 (S) h &=& R V^* \rho (S) \frac{(1_K - T)}{1 - t} V h \\
&=& \frac{R}{1 - t} V^* \left (S \otimes h + N - \lim_{\alpha} \sum_{i=1}^{n_\alpha} SS^h_{\alpha_i} \otimes h_{\alpha_i} + N \right )\\
&=& \frac{R}{1 - t} \left (\phi(S)h - \lim_{\alpha} \sum_{i=1}^{n_\alpha} \phi (S S^h_{\alpha_i})  h_{\alpha_i} \right ).
\end{eqnarray*}
Therefore, the range of $\phi_2$ is contained in $\pi(A)^{\prime\prime}$. This implies $\phi_2$ is a weak expectation of $\pi$.

For $\phi_1, \phi_2 \in WE(\pi)$, we get $\phi = t\phi_1 + (1-t)\phi_2$. Since $\phi$ is an extreme point, $\phi = \phi_1 = \phi_2$. Hence, $\phi_0 = t \phi_1 = t \phi$ 

Conversely, if $\phi_0 \in [0, \phi] \cap t\rm{WE}(\pi)$, then $\phi_0 = t\phi$ where, $0 < t < 1$. Suppose $\phi = \lambda\phi_1 + (1 - \lambda) \phi_2$, for $\phi_1, \phi_2 \in \rm{WE}(\pi)$ and $0 < \lambda < 1$. Then $\lambda\phi_1 \in [0, \phi] \cap \lambda\rm{WE}(\pi)$, hence $\lambda\phi_1 = \lambda\phi$ and similarly $(1 - \lambda) \phi_2 = (1- \lambda) \phi$. This implies $\phi_1 = \phi = \phi_2$.		
	
$(1)\Leftrightarrow (3)$: Assume that $[\rho(\pi(A))VH]$ is a faithful subspace for $\mathcal{C}^S_\phi$. Let $\phi = \lambda \phi_1 + (1 - \lambda) \phi_2$ for $\phi_1, \phi_2 \in \rm{WE}(\pi)$ and $0 < \lambda < 1$. Since $\lambda \phi_1 \leq \phi$, using Theorem \ref{rnd} we get, $\lambda \phi_1 = V^* T \rho V$ for some $T \in \rho(B(H))^\prime$ such that $0 \leq T \leq 1_K$. Moreover the operator $T$ satisfies equations (1) and (2) from Lemma \ref{ml1} for $t = \lambda$. Let $P$ be the projection of the Hilbert space $K$ onto $[\rho(\pi(A))VH]$. Let $a_1, a_2 \in A$ and $h_1, h_2, \in H$, 
\begin{eqnarray*}
\left \langle P T \rho(\pi(a_1)) V h_1, \hspace{5pt} \rho(\pi(a_2)) V h_2 \right \rangle &=& \left \langle T \rho(\pi(a_1)) V h_1, ~ \rho(\pi(a_2)) V h_2 \right \rangle \\
&=& \left \langle V^* T \rho(\pi(a_2)^*\pi(a_1)) V h_1, ~ h_2 \right \rangle \\
&=& \left \langle \lambda \pi(a_2)^*\pi(a_1) h_1, ~ h_2 \right \rangle \\
&=& \lambda \left \langle V^* \rho(\pi(a_2)^*\pi(a_1)) V h_1, ~ h_2 \right \rangle \\
&=&  \left \langle \lambda \rho(\pi(a_1)) V h_1, ~ \rho(\pi(a_2)) V h_2 \right \rangle
\end{eqnarray*}
Therefore, $PT = \lambda 1_{[\rho(\pi(A))VH]}$. This implies that $P(T - \lambda 1_K)_{[\rho(\pi(A))VH]}$ is $0$. Since, $[\rho(\pi(A))VH]$ is a faithful subspace of $K$, for $\mathcal{C}^S_\phi$ and $T - \lambda 1_K \in \mathcal{C}^S_\phi$ we get, $T = \lambda 1_K$. Hence, $\lambda \phi_1 = V^* T \rho V = \lambda \phi$, and $\phi_1 = \phi = \phi_2$.

Conversely, let us assume that $\phi$ is an extreme point of $\rm{WE}(\pi)$. Define a positive linear map $\mu: \rho(B(H))^\prime \rightarrow B([\rho(\pi(A))VH])$ by, $\mu (T) := PT_{[\rho(\pi(A))VH]}$. When $T \in \mathcal{C}^S_\phi$, $T$ is self-adjoint. So we may choose positive scalars $s, t$ such that $\frac{1}{4} 1_K \leq sT + t 1_K \leq \frac{3}{4} 1_K$. Then, as $\mu$ is positive, $\frac{1}{4} \mu(1_K) \leq \mu(sT + t 1_K) \leq \frac{3}{4} \mu(1_K)$. Now let $\mu(T) = 0$. So we get: $\frac{1}{4} 1_{[\rho(\pi(A))VH]} \leq  t 1_{[\rho(\pi(A))VH]}\leq \frac{3}{4} 1_{[\rho(\pi(A))VH]}$. This implies that $0 < t < 1$. Next, define: 		
\begin{equation*}
\phi_1 := V^*(sT + t 1_K) \rho V ~ \rm{and} ~ \phi_2 := V^*(1_K - (sT + t 1_K)) \rho V.
\end{equation*} 
Then from Theorem \ref{rnd} we get that $\phi_1$ and $\phi_2$ are completely positive maps which are less than $\phi$ and $\phi_1 + \phi_2 = \phi$. Also, for all $a\in A$: 
\begin{eqnarray*}
\phi_1(\pi(a)) &=& V^*(sT + t 1_K) \rho(\pi(a)) V \\
&=& V^* sT \rho(\pi(a)) V + V^* t 1_K \rho(\pi(a)) V \\
&=& V^*VV^* sT \rho(\pi(a)) V + V^* t 1_K \rho(\pi(a)) V \\
&=& 0 + t \pi(a) = t \pi(a).
\end{eqnarray*}
The last equality follows from the fact that $V$ is an isometry (since $\phi$ is unital) and $[VH]\subseteq [\rho(\pi(A))VH]$. Similarly, we have $\phi_2(\pi(a)) = (1 - t) \pi(a)$ for all $a \in A$. 

Now we have: $\phi_1 =  V^*(sT + t 1_K) \rho V$ and $T \in \mathcal{C}^S_\phi$. Let $T = \alpha_1 T_1 + \alpha_2 T_2$, where $\alpha_i \in [-1, 1]$ and $T_i\in S_\phi$ for $i = 1, 2$. Therefore we get, 
\begin{equation*}
\phi_1 = V^*sT \rho V + V^* t 1_K \rho V = V^*s \alpha_1 T_1 \rho V + V^* s \alpha_2 T_2 \rho V + V^* t 1_K \rho V
\end{equation*}
Since both $T_1, T_2$ satisfy the second equation in Lemma \ref{ml1} we get that the range of $\phi_1$ is in $\pi(A)^{\prime \prime}$. Similarly the range of $\phi_2$ is in $\pi(A)^{\prime \prime}$. Therefore, we have: $\frac{1}{t} \phi_1, \frac{1}{1 - t} \phi_2 \in \text{WE}(\pi)$.		
	
Now, recall that $\phi = \phi_1 + \phi_2$. Writing $\phi = t (\frac{1}{t} \phi_1) + (1 - t) (\frac{1}{1 - t} \phi_2)$ and using the hypothesis that $\phi$ is an extreme point, we get: $$\phi = \frac{1}{t} \phi_1 = \frac{1}{1 - t} \phi_2.$$. 
This gives us
\begin{equation*}
\phi_1= V^*(sT + t 1_K) \rho V = V^* t1_K \rho V = t\phi \leq \phi.
\end{equation*}
By the order preserving affine isomorphism of Theorem \ref{rnd}, we must have $sT + t 1_K = t 1_K$, which in turn implies that $T=0$. This shows that, the map $\mu$ is injective on $\mathcal{C}^S_\phi$ and therefore $[\rho(\pi(A))VH]$ is a faithful subspace of $K$ for $\mathcal{C}^S_\phi$.
	
$(1)\Leftrightarrow (4)$: Let $\phi(\cdot) = \langle \sigma(\cdot)e, e \rangle$ be as in Theorem \ref{pd}. Recall, from the discussion before Theorem \ref{prnd} that, we get a representation $\tilde{\sigma}: B(H) \rightarrow \mathcal{P}(X^\prime)$ by $\tilde{\sigma}(S) := \widetilde{\sigma(S)}$ for all $S \in B(H)$. Also from Theorem \ref{prnd}, we have $\phi(\cdot) = \langle \sigma(\cdot)e, e \rangle = \langle 1_{X^\prime} \tilde{\sigma}(\cdot)\hat{e}, \hat{e} \rangle$. Suppose for $T \in \mathcal{C}^P_\phi$,  $\langle T \tilde{\sigma}(\pi(a))\hat{e}, \hat{e} \rangle = 0$ for all $a \in A$ implies $T=0$. Then we show that $\phi$ is an extreme point of $\rm{WE}(\pi)$. If $\phi = \lambda \phi_1 + (1 - \lambda) \phi_2$ for $\phi_1, \phi_2 \in \rm{WE}(\pi)$ and $0 < \lambda < 1$, then $\lambda \phi_1 \leq \phi$. Therefore by Theorem \ref{prnd} we get $T_1 \in \tilde{\sigma}(\mathcal{B}(H))^\prime$ with $0 \leq T_1 \leq 1_{X^\prime}$ such that $\lambda \phi_1(S) = \langle T_1 \tilde{\sigma}(S)\hat{e}, \hat{e} \rangle$ for all $S \in B(H)$. Moreover $T_1$ satisfies the condition in Lemma \ref{ml2} for $t = \lambda$. Hence, for $a \in A$ we have:
\begin{equation*}
\lambda \phi_1(\pi(a)) = \langle T_1 \tilde{\sigma}(\pi(a))\hat{e}, \hat{e} \rangle = \lambda \pi(a) = \lambda \langle 1_{X^\prime} \tilde{\sigma}(\pi(a))\hat{e}, \hat{e} \rangle.
\end{equation*}
Since, this is true for all $a \in A$ we get $\langle (T_1 - \lambda 1_{X^\prime}) \tilde{\sigma}(\pi(a))\hat{e}, \hat{e} \rangle = 0$ for all $a \in A$. Since $T_1 - \lambda 1_{X^\prime} \in \mathcal{C}^P_\phi$, by hypothesis, we get  $T_1 = \lambda 1_{X^\prime}$. Therefore, $\lambda \phi_1(\cdot)= \langle \lambda \tilde{\sigma}(\cdot)\hat{e}, \hat{e} \rangle = \lambda \phi(\cdot)$. This implies $\phi_1 = \phi = \phi_2$.
	
Conversely, assume that $\phi$ is an extreme point of $\rm{WE}(\pi)$. Let $T \in \mathcal{C}^P_\phi$ be such that $\langle T \tilde{\sigma}(\pi(a))\hat{e}, \hat{e} \rangle = 0$ for all $a \in A$.
For $a\in A$ define $\mu_a: \tilde{\sigma}(B(H))^\prime \rightarrow \pi(A)^{\prime \prime}$ by $S \mapsto \langle S \tilde{\sigma}(\pi(a))\hat{e}, \hat{e} \rangle$. By assumption, we have $\mu_a(T) = 0$ for all $a \in A$. If $T \in \mathcal{C}^P_\phi$, then $T = T^*$. Choose positive scalars $s, t$ such that $\frac{1}{4} 1_{X^\prime} \leq sT + t 1_{X^\prime} \leq \frac{3}{4} 1_{X^\prime}$. For an arbitrary but fixed $a \in A$, consider the positive map $\mu_{a^*a}$ defined as before. So we get $$\frac{1}{4} \mu_{a^*a}(1_{X^\prime}) \leq \mu_{a^*a}(sT + t 1_{X^\prime}) \leq \frac{3}{4} \mu_{a^*a}(1_{X^\prime}).$$ Since, $\mu_{a^*a}(T) = 0$, we have: 
\begin{eqnarray*}
\frac{1}{4} \langle \tilde{\sigma}(\pi(a))\hat{e}, \tilde{\sigma}(\pi(a))\hat{e} \rangle &\leq&  t \langle \tilde{\sigma}(\pi(a))\hat{e}, \tilde{\sigma}(\pi(a))\hat{e} \rangle \\
&\leq& \frac{3}{4} \langle \tilde{\sigma}(\pi(a))\hat{e}, \tilde{\sigma}(\pi(a))\hat{e} \rangle.
\end{eqnarray*} 
This implies that $0 < t < 1$. Define $\phi_i: B(H) \rightarrow \pi(A)^{\prime \prime}$ for $i=1, 2$ by
\begin{eqnarray*}
\phi_1(\cdot) &:=& \langle (sT + t 1_{X^\prime}) \tilde{\sigma}(\cdot) \hat{e}, \hat{e} \rangle  \hspace{5pt} \rm{and} \\
\phi_2(\cdot) &:=& \langle (1_{X^\prime} - (sT + t 1_{X^\prime})) \tilde{\sigma}(\cdot) \hat{e}, \hat{e} \rangle .
\end{eqnarray*}
Then by using Theorem \ref{prnd}, we have $0 \leq \phi_1, \phi_2 \leq \phi$ and $\phi_1 + \phi_2 = \phi$. Further, for all $a \in A$ we have 
\begin{equation*}
\phi_1(\pi(a)) = t \pi(a) ~ \rm{~and~} ~ \phi_2(\pi(a)) = (1-t) \pi(a)
\end{equation*}
as well as that, the ranges of both $\phi_1$, $\phi_2$ are in $\pi(A)^{\prime \prime}$. Therefore, $t^{-1} \phi_1$ and $(1-t)^{-1}\phi_2$ belong to $\rm{WE}(\pi)$. Write $\phi = t (\frac{1}{t} \phi_1) + (1 - t) (\frac{1}{1 - t} \phi_2)$. As $\phi$ is an extreme point, $\phi = \frac{1}{t} \phi_1 = \frac{1}{1 - t} \phi_2$. This gives
\begin{equation*}
\phi_1(\cdot) = \langle (sT + t 1_{X^\prime}) \tilde{\sigma}(\cdot) \hat{e}, \hat{e} \rangle = \langle t 1_{X^\prime} \tilde{\sigma}(\cdot) \hat{e}, \hat{e} \rangle.
\end{equation*}
By the affine order isomorphism from Theorem \ref{prnd} we get $T = 0$.
\end{proof}

Recall from Section \ref{sec;I} that the family of continuous linear functionals on $E = \rm{CB}(B(H), \adcom)$ given by:
\begin{equation*}
\{ f^S_{(h, k)}: E  \rightarrow \bbC ~ | ~ f^S_{(h, k)}(\phi) := \la \phi(S)h, k \ra; S \in B(H), h,k \in H, \phi \in E \}.
\end{equation*}
separates the points in $X = \rm{WE}(\pi)$. Equipped with the precise description of the extreme points of the compact convex set $\rm{WE}(\pi)$ from Theorem \ref{mt2} and using the separating family of functionals described above, we can now complete the barycentric description of an element in $\rm{WE}(\pi)$ by an application of the classical Choquet-Bishop-de Leeuw Theorem.

\begin{theorem}[barycentric decomposition of a weak expectation]
For every $\phi \in \rm{WE}(\pi)$, there exists a probability measure $\mu_\phi$ which vanishes on every Baire set of $\rm{WE}(\pi)$ disjoint from the set extreme points of  $\rm{WE}(\pi)$, i.e. pseudo supported on the set of extreme points of $\rm{WE}(\pi)$, and is given by
\begin{equation*}
f^S_{(h, k)}(\phi) = \int_{\rm{WE}(\pi)} f^S_{(h, k)} \, \mathrm{d} \mu_\phi 
\end{equation*}
for every $S \in B(H)$ and $h, k \in H$.
\end{theorem}  	

%%%%%%%%%%%%%%%%%%%%%%%%%%% bibliography %%%%%%%%%%%%%%%%%%%%%%%%%
%\newpage

\section*{Acknowledgement} 
The first named author is partially supported by Science and Engineering Board (DST, Govt. of India) grant no. ECR/2018/000016 and the second named author is supported by CSIR PhD scholarship award letter no. 09/1020(0142)/2019-EMR-I. 

\subsection*{Declaration}
The authors declare that there is no conflict of interest.

\end{document}